\documentclass[a4paper,11pt]{amsart}
\usepackage{amsfonts}
\usepackage{amssymb}
\usepackage{amsmath}
\usepackage{mathrsfs}
\usepackage{amsmath,amssymb,amsthm,latexsym,amscd,mathrsfs}
\usepackage{indentfirst}
\usepackage{graphicx}
\usepackage{booktabs}
\usepackage{array}
\usepackage{chemarrow}
\newcommand{\ROM}[1]{\mathrm{\uppercase\expandafter{\romannumeral#1}}}
\theoremstyle{definition}

\newtheorem{thm}{Theorem}[section]
\newtheorem{lem}{Lemma}[section]

\newtheorem{rem}{Remark}[section]

\newtheorem{ack}{Acknowledgements}   
\makeatletter


\title[New examples of Willmore submanifolds in the unit sphere, II]{\textbf{New examples of Willmore submanifolds in the unit sphere via isoparametric functions, II}}
\author[C. Qian]{Chao Qian}\address{School of Mathematical Sciences, Laboratory of Mathematics and Complex Systems, Beijing Normal
University, Beijing 100875, China}\email{qianchao\_1986@163.com}
\author[Z.Z.Tang]{Zizhou Tang}\address{School of Mathematical Sciences, Laboratory of Mathematics and Complex Systems, Beijing Normal
University, Beijing 100875, China}\email{zztang@bnu.edu.cn}
\thanks {$^{\dag}$ the corresponding author}
\thanks {The project is partially supported by the NSFC ( No.11071018 ) and the Program for Changjiang Scholars and Innovative
Research Team in University.}
\author[W. J. Yan]{Wenjiao Yan$^{\dag}$}
\address{School of Mathematical Sciences, Laboratory of Mathematics and Complex Systems, Beijing Normal
University, Beijing 100875, China} \email{wjyan@mail.bnu.edu.cn}

\date{}
\begin{document}

\maketitle
\begin{center}
Dedicated to Professor Banghe Li on his 70-th birthday.
\end{center}

\begin{abstract}
This paper is a continuation and wide extension of \cite{TY}. In the
first part of the present paper, we give a unified geometric proof
that both focal submanifolds of every isoparametric hypersurface in
spheres with four distinct principal curvatures are Willmore. In the
second part, we completely determine which focal submanifolds are
Einstein except one case.
\end{abstract}

\section{Introduction}

Let $x:M^n \rightarrow S^{n+p}$ be an immersion from an
$n$-dimensional compact manifold to an $(n + p)$-dimensional unit
sphere $S^{n+p}$. Denote by $h$ the second fundamental form of $x$,
$S$ the norm square of $h$, and $H$ the norm of the mean curvature
vector, respectively. Then $M^n$ is called a \emph{Willmore
submanifold} in $S^{n+p}$ if it is an extremal submanifold of the
Willmore functional, which is a conformal invariant (\emph{cf.}
\cite{Wan1}):
$$W(x)=\int_{M^n} (S-nH^2)^{\frac{n}{2}}dv.$$

An equivalent condition for $M^n$ to be Willmore was given in \cite{GLW}, \cite{PW}.
In particular, when $M^n$ is a minimal submanifold in $S^{n+p}$ with constant $S$,
under a field of local orthonormal basis
$\{e_A\}$ $(1 \leq A \leq n + p)$ for $TS^{n+p}$, in which $\{e_i\}\in TM$ $(1\leq i\leq n)$
and $\{e_{\alpha}\}\in T^{\perp}M$ $(n+1\leq \alpha \leq n+p)$, their criterion
for Willmore reduces to be:
\begin{equation}\label{reduced Willmore}
for ~any  ~\alpha, ~~\sum_{i,j=1}^{n} R_{ij}h^{\alpha}_{ij}=0,
\end{equation}
where $R_{ij}$ is the Ricci tensor of $M^n$, $h^{\alpha}_{ij}$ is the component of $h$
with respect to $e_i, e_j$ and $e_{\alpha}$.

Based on the reduced criterion (\ref{reduced Willmore}), in
conjunction with the fact that the focal submanifolds of an
isoparametric hypersurface in $S^{n+1}$ are minimal submanifolds of
$S^{n+1}$ with constant $S$, we establish one of our main results as
follows
\begin{thm}\label{thm1}
{\itshape Both focal submanifolds of every isoparametric
hypersurface in unit spheres with four distinct principal curvatures
are Willmore.}
\end{thm}

\begin{rem}
This theorem extends widely the main result in \cite{TY}, where one
focal submanifold of every isoparametric hypersurface of FKM-type
was shown to be Willmore.
\end{rem}

We need some preliminaries on isoparametric hypersurfaces. It is
well known that an isoparametric hypersurface in a complete
Riemannian manifold $N$ always comes as a family of parallel
hypersurfaces, which are level hypersurfaces of an \emph{
isoparametric function} $f$, that is, a non-constant smooth function
on $N$ satisfying
\begin{equation}\label{ab}
\left\{ \begin{array}{ll}
|\nabla f|^2= b(f),\\
~~~~\Delta f~~=a(f),
\end{array}\right.
\end{equation}
where $\nabla f$ and $\Delta f$ are the gradient and Laplacian of
$f$ on $N$, respectively, $b$ and $a$ are smooth and continuous functions
on $\mathbb{R}$, respectively. The preimage of the global
maximum (\emph{resp}. minimum) of an isoparametric function $f$ is
called the \emph{focal variety} of $f$, denoted by $M_{+}$
(\emph{resp}. $M_{-}$), if nonempty. A fundamental structural result
claimed by \cite{Wa87}, proved in details by \cite{GT}, asserts that
each focal variety of an isoparametric function is a minimal
submanifold of $N$.

The equations in (\ref{ab}) mean that the parallel level
hypersurfaces have constant mean curvatures. As is well known, an
isoparametric hypersurface $M^n$ in the unit sphere $S^{n+1}$
actually has constant principal curvatures. Let $g$ be the number of
distinct principal curvatures, which are denoted by $k_i$ ($k_1 >$
... $>k_g$) with multiplicity $m_i$ ($i=1,...,g$). A remarkable
result proved by M{\"u}nzner \cite{Mun} states that $m_i=m_{i+2}$
(subscripts mod $g$) and the isoparametric function $f$ must be the
restriction to $S^{n+1}$ of a homogeneous polynomial $F:
\mathbb{R}^{n+2} \rightarrow \mathbb{R}$ of degree $g$ satisfying
the \emph{Cartan-M{\"u}nzner equations}:
\begin{equation}\label{C-M}
\left\{ \begin{array}{ll}
|\nabla F|^2= g^2|x|^{2g-2}, \\
~~~~\Delta F~~=\frac{m_2-m_1}{2}g^2|x|^{g-2},
\end{array}\right.
\end{equation}
where $\nabla F$ and $\Delta F$ are the gradient and Laplacian of
$F$ on $\mathbb{R}^{n+2}$. The polynomial $F$ is called the
\emph{Cartan-M{\"u}nzner polinomial}, its restriction
$f=F|_{S^{n+1}}$ takes values in $[-1, 1]$ on $S^{n+1}$. The focal
submanifolds $M_+:=f^{-1}(1)$ and $M_-:=f^{-1}(-1)$ are in fact
minimal submanifolds of $S^{n+1}$ with respective codimensions
$m_1+1$ and $m_2+1$, whose second fundamental forms are both of
constant length (\emph{cf.} \cite{CR}).

As a corollary of the criterion (\ref{reduced Willmore}), every
n-dimensional Einstein manifold minimally immersed in the unit
sphere $S^{n+p}$ is a Willmore submanifold. Then a natural problem
arises: is every focal submanifold Einstein? According to the
concluding remark of \cite{TY}, there are only few focal
submanifolds $M_+$ of FKM-type being possibly Einstein. As another
main result of this paper, we give a complete resolution of this
problem for the focal submanifolds of every isoparametric
hypersurface in $S^{n+1}$ with four distinct principal curvatures,
except the open case $(m_1, m_2)=(7,8)$.

To state Theorem \ref{thm2} clearly, we recall the construction of
isoparametric functions of FKM-type. For a symmetric Clifford system
$\{P_0,\cdots,P_m\}$ on $\mathbb{R}^{2l}$, \emph{i.e.} $P_i$'s are
symmetric matrices satisfying $P_iP_j+P_jP_i=2\delta_{ij}I_{2l}$,
Ferus, Karcher and M\"{u}nzner (\cite{FKM}) constructed a polynomial
$F$ (called \emph{FKM-type isoparametric polynomial}) of degree $4$
on $\mathbb{R}^{2l}$:
\begin{eqnarray}\label{FKM isop. poly.}
&&\qquad F:\quad \mathbb{R}^{2l}\rightarrow \mathbb{R}\nonumber\\
&&F(x) = |x|^4 - 2\displaystyle\sum_{i = 0}^{m}{\langle
P_ix,x\rangle^2},
\end{eqnarray} which satisfies the Cartan-M{\"u}nzner equations. Moreover, it is easy to verify that $f=F|_{S^{2l-1}}$ satisfies (\emph{cf.} \cite{GTY}):
\begin{equation}\label{ab in Sn}
\left\{ \begin{array}{ll}
|\nabla f|^2= 16 (1-f^2),\\
~~~\Delta f~~=8(m_2-m_1)-4(2l+2)f,
\end{array}\right.
\end{equation}
where $m_1=m$, $m_2=l-m-1$. Thus by definition, $f$ is an
\emph{isoparametric function} on $S^{2l-1}$.

According to \cite{CCJ}, \cite{Imm} and \cite{CHI}, all
isoparametric hypersurfaces in spheres with four distinct principal
curvatures are of FKM-type, except for the cases $(m_1, m_2)=(2,2)$
and $(4,5)$, and except possibly for cases with multiplicities
$(7,8)$, which has not been classified yet. As for the isoparametric
hypersurfaces with $(m_1, m_2)=(2,2)$ or $(4,5)$, they must be
homogeneous and thus unique (\emph{cf.} \cite{OT}, \cite{CHI}). We
are now ready to state the following

\begin{thm}\label{thm2}
{\itshape For the focal submanifolds of an isoparametric
hypersurface in $S^{n+1}$ with four distinct principal curvatures,
we have
\begin{itemize}
\item[(i)] All the $M_-$ of FKM-type are not Einstein; the $M_+$ of
           FKM-type is Einstein if and only if it is diffeomorphic to $Sp(2)$ in the homogeneous case with
           $(m_1,m_2)=(4,3)$.
\item[(ii)] In the case $(m_1,m_2)=(2,2)$, the focal submanifold
           diffeomorphic to $\widetilde{G}_2(\mathbb{R}^5)$ is Einstein, while the other one
           diffeomorphic to $\mathbb{C}P^3$ is not.
\item[(iii)]In the case $(m_1,m_2)=(4,5)$, both focal submanifolds are not Einstein.
\end{itemize}
}
\end{thm}

\begin{rem}
$(1)$. In the FKM-family, there are two incongruent examples corresponding to $(m_1,m_2)=(4,3)$, one is homogeneous
while the other is not. It is surprising that only the focal submanifold $M_+$ of the homogenous case is Einstein.

$(2)$. For g=4, we provide a complete determination for which focal
submanifolds are Einstein except the case $(m_1, m_2)=(7,8)$, which
has not been classified yet.
Fortunately, we are able to show that
for all the known examples with $(m_1, m_2)=(7,8)$ (in fact, three examples of FKM-type), the focal
submanifolds are not Einstein.
\end{rem}

\section{Isoparametric foliations}
\subsection{Preliminaries}
Let $M^n$ be an isoparametric hypersurface with four distinct
principal curvatures in the unit sphere $S^{n+1}$, and $F$ be the
corresponding Cartan-M{\"u}nzner polynomial. In the current
discussion, we focus only on the focal submanifold
$M_+=F^{-1}(1)\cap S^{n+1}$, since we can change $F$ to $-F$ so if
necessary.

In virtue of M\"{u}nzner, $M^n$ can be regarded as a unit normal
sphere bundle $UN_+$ over $M_+$. In addition, at any point $x\in
M_+$, the principal curvatures of the shape operator with respect to
any unit normal vector are $0,\ 1,\ -1$, with the corresponding
multiplicities $m_1,\ m_2$ and $m_2$.

Following \cite{CCJ}, let $(x, n_{0})\in UN_+$ be points in a small
open set, where $x\in M_+$ and $n_0$ is a unit normal vector of
$M_+$ at $x$. Adopting the following index ranges
\begin{eqnarray*}
&& i, j, k \in \{1,\cdots , n=2(m_1+m_2)\},\,\,\,\,\, a, b, c \in \{1,\cdots, m_1\},\\
&& p, q, r \in \{m_1+1,\cdots, 2m_1\},\,\,\,\,\, \alpha, \beta, \gamma \in \{2m_1+1,\cdots, 2m_1+m_2\},\\
&& \mu, \nu, \sigma \in \{2m_1+m_2+1,\cdots, 2m_1+2m_2\},
\end{eqnarray*}
we choose a smooth orthonormal frame $\{n_a\},\ \{e_p\},\
\{e_{\alpha}\},\ \{e_{\mu}\}$ of $UN_+$ in such a way that $\{n_a\}$
are tangent to the unit normal sphere at $n_0$, and $\{e_p\},\
\{e_{\alpha}\},\ \{e_{\mu}\}$ are respectively the basis vectors of
the eigenspaces $V_0,\ V_+,\ V_{-}$ of the shape operator $S_{n_0}$.

Since each frame vector can be regarded as a smooth map from $UN_+$
to $\mathbb{R}^{n+2}$, using the Einstein summation convention, we
have
\begin{eqnarray}\label{moving equation}
&&dx=\omega ^pe_p + \omega ^{\alpha }e_{\alpha }+ \omega ^{\mu }e_{\mu },\nonumber\\
&&d n_0=\omega ^bn_b - \omega ^{\alpha }e_{\alpha }+ \omega ^{\mu }e_{\mu }, \nonumber\\
&&d n_a=-\omega ^an_0+ \theta _{a}^{b}n_b + \theta _{a}^{q}e_q+  \theta _{a}^{\alpha }e_{\alpha }+ \theta _{a}^{\mu  }e_{\mu  },\\
&&de_p=-\omega ^px+ \theta _{p}^{b}n_b + \theta _{p}^{q}e_q+  \theta _{p}^{\alpha }e_{\alpha }+ \theta _{p}^{\mu  }e_{\mu  }, \nonumber\\
&&de_{\alpha }=-\omega ^{\alpha }x+ \omega ^{\alpha }n_0+\theta _{\alpha }^{b}n_b + \theta _{\alpha}^{q}e_q+  \theta _{\alpha}^{\beta }e_{\beta }+ \theta _{\alpha}^{\mu  }e_{\mu }, \nonumber\\
&&de_{\mu  }=-\omega ^{\mu  }x- \omega ^{\mu }n_0+\theta _{\mu  }^{b}n_b + \theta _{\mu}^{q}e_q+  \theta _{\mu}^{\alpha }e_{\alpha }+ \theta _{\mu}^{\nu   }e_{\nu },\nonumber
\end{eqnarray}
where
\begin{eqnarray}\label{connection}
&&\theta _{a}^{p}=\sum_{\alpha }F_{pa}^{\alpha}\omega ^{\alpha }- \sum_{\mu }F_{pa}^{\mu}\omega ^{\mu },  \ \theta _{a}^{\alpha }=\sum_{p }F_{pa}^{\alpha}\omega ^{p}- 2\sum_{\mu }F_{\alpha a}^{\mu}\omega ^{\mu }, \nonumber\\
&&\theta _{p}^{\alpha }=\sum_{a }F_{pa}^{\alpha}\omega ^{a }- 2\sum_{\mu }F_{\alpha p}^{\mu}\omega ^{\mu }, \ \theta _{a}^{\mu}=-\sum_{p }F_{pa}^{\mu}\omega ^{p }- 2\sum_{\alpha }F_{\alpha a}^{\mu}\omega ^{\alpha  }, \\
&&\theta _{p}^{\mu }=\sum_{a }F_{pa}^{\mu}\omega ^{a }+ 2\sum_{\alpha  }F_{\alpha p}^{\mu}\omega ^{\alpha },\ \theta _{\alpha }^{\mu }=\sum_{a }F_{\alpha a}^{\mu}\omega ^{a }+ \sum_{p}F_{\alpha p}^{\mu}\omega ^{p}.\nonumber
\end{eqnarray}

Combining the third equation in (\ref{moving equation}) with
(\ref{connection}), we obtain an explicit expression of
$S_a:=S_{n_a}$ in terms of $F_{\alpha a}^{\mu},\ F_{\alpha p}^{\mu},
\ F_{pa}^{\mu},\ F_{pa}^{\alpha}$ defined above and the orthonormal
coframe field $\omega ^{p},\ \omega ^{\alpha },\ \omega ^{\mu}$:
\begin{eqnarray}\label{shape operator}
\qquad S_a=(2F_{\alpha a}^{\mu}e_{\mu}-F_{pa}^{\alpha }e_p)\omega ^{\alpha }+ (2F_{\alpha a}^{\mu}e_{\alpha}+F_{pa}^{\mu }e_p)\omega ^{\mu}+(-F_{p a}^{\alpha }e_{\alpha }+F_{pa}^{\mu}e_{\mu})\omega ^{p},
\end{eqnarray}

Define linear operators ( \cite{CCJ} )
\begin{eqnarray}\label{ABC}
&&A_a=2F_{\alpha a}^{\mu}e_{\alpha }\omega ^{\mu}: V_- \rightarrow V_+, \nonumber\\
&&B_a=-F_{p a}^{\alpha}e_{\alpha }\omega ^{p}: V_0 \rightarrow V_+,  \\
&&C_a=F_{p a}^{\mu}e_{\mu }\omega ^{p}: \quad V_0 \rightarrow V_-,\nonumber
\end{eqnarray}
and their transposes
\begin{eqnarray*}
&&^{t}A_a=2F_{\alpha a}^{\mu}e_{\mu}\omega ^{\alpha }: V_+ \rightarrow V_-, \\
&&^tB_a=-F_{p a}^{\alpha}e_{p}\omega ^{\alpha }: V_+ \rightarrow V_0,  \\
&&^tC_a=F_{p a}^{\mu}e_{p }\omega ^{\mu}:\quad V_- \rightarrow V_0.
\end{eqnarray*}
With respect to the orthogonal direct sum decomposition $$T_xM_+=V_+\oplus V_-\oplus V_0=Span\{e_{\alpha}\}\oplus Span\{e_{\mu}\}\oplus Span\{e_{p}\},$$
the shape operator $S_a$ has the block form
$$S_a=\left(
\begin{array}{ccc}
0&A_a&B_a\\
^tA_a&0&C_a\\
^tB_a&^tC_a&0
\end{array}
\right).$$

\subsection{\emph{Proof of Theorem \ref{thm1}}.}
\quad Let $M^n$ be an isoparametric hypersurface with four distinct
principal curvatures in the unit sphere $S^{n+1}$, and let $F$ be its
Cartan-M{\"u}nzner polynomial. We only make the proof for $M_+$, as
the proof for $M_-$ is analogous.

Take the same notations as in Section 2.1. Based on the principal
decomposition $V_+$, $V_0$ and $V_-$ of the tangent space $TM_+$, we
simplify the criterion (\ref{reduced Willmore}) for Willmore
submanifolds to be:
\begin{equation}\label{M+ willmore criterion}
\sum_{\alpha }Ric(e_{\alpha })=\sum_{\mu }Ric(e_{\mu }).
\end{equation}

At first, from Gauss equation, we derive that
\begin{eqnarray*}\label{Ricci M+}
Ric(e_{\alpha }) &=& \sum _{t=m_1+1}^{2m_1+2m_2} \overline{R}(e_t, e_{\alpha} , e_t, e_{\alpha })+ \sum_{t=m_1+1}^{2m_1+2m_2}\langle h(e_t, e_t),h(e_{\alpha }, e_{\alpha }) \rangle-\sum_{t=m_1+1}^{2m_1+2m_2}|h(e_t, e_{\alpha })|^2 \nonumber\\
&=& (m_1+2m_2-1)+\sum_{t=m_1+1}^{2m_1+2m_2}\langle h(e_t, e_t),h(e_{\alpha }, e_{\alpha }) \rangle-\sum_{t=m_1+1}^{2m_1+2m_2}|h(e_t, e_{\alpha })|^2
\end{eqnarray*}
where $\overline{R}$ is the curvature tensor of $S^{n+1}$.

Immediately, the minimality of $M_+$ gives rise to
$$\sum_t\langle h(e_t, e_t),h(e_{\alpha }, e_{\alpha }) \rangle =0.$$

Additionally, another straightforward calculation depending on the formula (\ref{shape operator})
leads us to the last item in the expression of $Ric(e_{\alpha})$ as follows
\begin{eqnarray}
\sum_t|h(e_t, e_{\alpha })|^2&=& \sum_t|\sum_a\langle h(e_t, e_{\alpha }),n_a\rangle n_a+ \langle h(e_t, e_{\alpha }),n_0\rangle n_0|^2, \nonumber\\
      &=& \sum_{a,t}\langle S_a e_{\alpha },e_t \rangle^2+ \sum_t\langle S_0 e_{\alpha },e_t \rangle^2,\nonumber\\
      &=& \sum_a|S_a e_{\alpha }|^2 +1,\nonumber\\
      &=& 4\sum_{a,\mu}(F_{\alpha a}^{\mu})^2+ \sum_{a,p}(F_{pa}^{\alpha })^2+1.\nonumber
\end{eqnarray}

In summary, we arrive at
$$\sum_{\alpha }Ric(e_{\alpha })=m_2(m_1+2m_2-2)-4\sum_{a, \alpha, \mu}(F_{\alpha a}^{\mu})^2- \sum_{a,p, \alpha }(F_{pa}^{\alpha })^2.$$
Similarly,
$$\sum_{\mu}Ric(e_{\mu })=m_2(m_1+2m_2-2)-4\sum_{a, \alpha, \mu}(F_{\alpha a}^{\mu})^2- \sum_{a,p, \mu}(F_{pa}^{\mu })^2.$$
Therefore, the equality $\displaystyle\sum_{\alpha }Ric(e_{\alpha
})=\displaystyle\sum_{\mu }Ric(e_{\mu })$ holds if and only if
\begin{equation}\label{willmore equivalent}
\sum_{a,p, \alpha }(F_{pa}^{\alpha })^2=\sum_{a,p, \mu}(F_{pa}^{\mu })^2.
\end{equation}
As claimed in Lemma 49 of \cite{CCJ}, for any choice of
$a\in\{1,...,m_1\}$, there is always an orthonormal basis in $V_+$
and an orthonormal basis in $V_-$ such that relative to these bases,
$B_a=C_a$. In other words, by the definition (\ref{ABC}) of $B_a$
and $C_a$, the equality (\ref{willmore equivalent}) holds.

Now, the proof of Theorem \ref{thm1} is complete.

$\hfill{\square}$

\section{FKM-type isoparametric polynomials}
In this section, we give a detailed study on the focal submanifolds
of FKM-type and show a complete determination for which focal
submanifolds are Einstein in this situation, providing a proof of 1)
in Theorem \ref{thm2}.

\subsection{$M_-$ of FKM-type}
Recalling the Cartan-M{\"u}nzner polynomial constructed by Ferus, Karcher and M\"{u}nzner (\cite{FKM}):
$$F(x) = |x|^4 - 2\displaystyle\sum_{i = 0}^{m}{\langle
P_ix,x\rangle^2},$$
we find that
\begin{eqnarray*}
M_- &=&F^{-1}(-1)\cap S^{2l-1}\\
&=& \{x\in S^{2l-1} |\  there \ exists \ P \in \Sigma(P_0,\cdots,P_m) \ with\  Px=x\},
\end{eqnarray*}
where $\Sigma(P_0,\cdots,P_m)$ is the unit sphere in
$Span\{P_0,\cdots,P_m\}$, which is called \emph{the Clifford
sphere}.
Notice that for any $P\in \Sigma(P_0,\cdots,P_m)$, we have $P^2=I$
and $\mathrm{Trace } ~P=0$. Thus the eigenvalues of $P$ must be $\pm
1$, with the same multiplicities. Denoting the corresponding
eigenspaces by $E_+(P)$ and $E_-(P)$ respectively, we can decompose
$\mathbb{R}^{2l}$ as
$$\mathbb{R}^{2l}=E_+(P)\oplus E_-(P).$$

Let $y \in M_-$ and $P\in \Sigma(P_0,\cdots,P_m)$ with $Py=y$. Define
$$\Sigma_P:=\{Q\in \Sigma(P_0,\cdots,P_m)|\  \langle P, Q \rangle:=\frac{1}{2l}\mathrm{Trace}(PQ)=0 \}, $$
which is the equatorial sphere of $\Sigma(P_0,\cdots,P_m)$ orthogonal to $P$.
In this way, we have a decomposition of the tangent space $TM_-$
with respect to the eigenspaces of the shape operator.

\vspace{2mm} \noindent \textbf{Lemma (\cite{FKM})}\,\, {\itshape The
principal curvatures of the shape operator $S_N$ with respect to any
unit normal vector $N\in T^{\perp}_yM_-$ are $0, 1$, and $-1$, with
the corresponding eigenspaces $Ker(S_N)$, $E_+(S_N)$, $E_-(S_N)$ as
follows:
\begin{eqnarray}\label{eigenspacesM-}
Ker(S_N)&=& \{v\in E_+(P)| \ v \bot y,\  v\bot \Sigma_P N\},\nonumber\\
E_+(S_N)&=& \mathbb{R}\Sigma_P (y+N),\\
E_-(S_N)&=& \mathbb{R}\Sigma_P (y-N).\nonumber
\end{eqnarray}
Moreover, $$\dim Ker(S_N)=l-m-1,\ \dim E_+(S_N)=\dim E_-(S_N)=m.$$\hfill $\Box$
}

To facilitate the description, in front of the proof of 1) in
Theorem \ref{thm2}, we state the following lemma, which is a direct
corollary of the Gauss equation for the minimal submanifold.

\begin{lem}\label{Einstein}
\emph{Let $M^n$ be an $n$-dimensional submanifold minimally immersed
in an $(n + p)$-dimensional unit sphere $S^{n+p}$. Then $M$ is
Einstein if and only if for arbitrary $x\in M$ and orthonormal basis
$\{N_{\alpha}\}$ of $T^{\perp}_xM$,
$$\sum_{\alpha=n+1}^{n+p}|S_{N_{\alpha}}(X)|^2\ is\  constant, \quad \forall X\in T_xM \ with \ |X|=1,$$
where $S_{N_{\alpha}}$ is the shape operator with respect to the
normal vector $N_{\alpha}$.}
\end{lem}


With all these preparations, we are in a position to prove that the focal submanifold $M_-$ of FKM-type is not Einstein.
\vspace{3mm}

\noindent\emph{\textbf{Proof of 1) in Theorem \ref{thm2} for
$M_-$:}} \quad Given $y \in M_-$ and $P\in \Sigma(P_0,\cdots,P_m)$
with $Py=y$, let $\{N_{\alpha}\}_{\alpha=1}^{l-m}$  be an
orthonormal basis for $T^{\perp}_yM_-$. For any $X\in T_yM_-$ with
$|X|=1$, there is a decomposition
$$X=X^0_{\alpha}+X^+_{\alpha}+X^-_{\alpha}\in Ker(S_{N_{\alpha}})\oplus E_+(S_{N_{\alpha}})\oplus E_-(S_{N_{\alpha}})$$
with respect to $N_{\alpha}$ by the lemma of \cite{FKM} mentioned
before. For any $Q\in\Sigma_P$, we define $Y=Qy$, and decompose it
as
$$Y=\frac{1}{2}Q(y+N_{\alpha})+\frac{1}{2}Q(y-N_{\alpha})\in E_+(S_{N_{\alpha}})\oplus E_-(S_{N_{\alpha}})$$
for any $\alpha$. Thus $Y\ \bot\ \cup_{N_{\alpha}\in
T^{\perp}_yM_-}Ker(S_{N_{\alpha}})$, which implies immediately that
$$ \sum_{\alpha=1}^{l-m}|S_{N_{\alpha}}(Y)|^2=l-m-\sum_{\alpha=1}^{l-m}|Y^{0}_{\alpha}|^2=l-m.$$
On the other hand, choosing $Z\in Ker(S_{N_1})$ with $|Z|=1$, we
have
$$ \sum_{\alpha=1}^{l-m}|S_{N_{\alpha}}(Z)|^2=l-m-\sum_{\alpha=1}^{l-m}|Z^{0}_{\alpha}|^2\leq l-m-1.$$
Therefore, $M_-$ is not Einstein by Lemma \ref{Einstein}.
$\hfill{\square}$

\subsection{$M_+$ of FKM-type}
This subsection will be committed to proving 1) of Theorem \ref{thm2} for the focal submanifold $M_+$ of FKM-type.

We start with a description of the normal space of the focal
submanifold
$$M_+=\{~x\in S^{2l-1}~|~\langle P_0x,x\rangle=...=\langle P_mx,x\rangle=0~\}.$$
As pointed out by \cite{FKM}, the normal space at $x\in M_+$ is
\begin{equation}
T^{\perp}_xM_+=\{~Px~|~P\in\mathbb{R}\Sigma (P_0,...,P_m)~\}.
\end{equation}
Following \cite{TY}, $\forall X \in T_xM_+$ with $|X|=1$, let
$\{X=e_1$, $e_2,...,e_{2l-m-2}\}$ be an orthonormal basis for
$T_xM_+$, and $\{P_0x,...,P_mx\}$ an orthonormal basis for
$T^{\perp}_xM_+$. Then from the Gauss equation and properties of
Clifford system $\{P_0,\cdots,P_m\}$, we derive the Ricci curvature
of $X$:
\begin{equation}\label{Ricci}
Ric (X)= 2(l-m-2) + 2\sum_{i, j =0, i < j}^{m}\langle X, P_{i}P_{j}x\rangle^2.
\end{equation}
As established by \cite{TY} in their concluding remark, a sufficient
condition for $M_+$ not
 to be Einstein can be stated as
$$\dim~M_+> \dim~Span\{P_{i}P_{j}x~|~ i, j=0,1,...,m, i<j\}.$$
Moreover, \cite{TY} reveals that the focal submanifold $M_+$ of
FKM-type is not Einstein, except possibly for those with
multiplicities in one of the following pairs
\begin{equation}\label{maybe Einstein}
  (m_1, m_2)=(m, l-m-1)=(4,3),\ (5,2),\ (6,1),\ (7,8),\ (8,7),\ (9,6),\ (10,21).
\end{equation}

As asserted by \cite{FKM}, the isoparametric families of FKM-type
with multiplicities $(5,2),\ (6,1)$ are congruent to those with
multiplicities $(2,5),\ (1,6)$.
Hence, both $M_+$ in the cases $(m_1, m_2)=(5,2),$ and $(6,1)$ are
not Einstein by the proved result for $M_-$ in last subsection.

Consequently, the cases with $(m_1, m_2)=(4,3),\ (7,8)$ ,$\ (8,7)$
,$\ (9,6),\ (10,21)$ are left to our consideration. Since the
$(4,3)$ case is the most amazing one, we will firstly deal with this
case. \vspace{1mm}

\subsubsection{\textbf{The (4,3) case.}}
In this case, $m=4$ and $l=8=2\delta (4)$, where $\delta (m)$ is the dimension of irreducible representation of $C_{m-1}$. According to \cite{FKM}, there are two examples of FKM-type isoparametric polynomials with multiplicities $(4,3)$, which are distinguished by
an invariant
\begin{equation*}
\mathrm{Trace} (P_0P_1P_2P_3P_4)=2q\delta (4),~~with~~ q\equiv2\
mod\ 2.
\end{equation*}
Noticing that $P_0P_1P_2P_3P_4$ is a symmetric orthogonal matrix on
$\mathbb{R}^{16}$, we divide the proof into two parts. \vspace{2mm}

1). $q=2$. On this condition, $P_0P_1P_2P_3P_4$ is a symmetric
orthogonal matrix on $\mathbb{R}^{16}$ with
$\mathrm{Trace}(P_0P_1P_2P_3P_4)=16$. Thus it follows easily that
$P_0P_1P_2P_3P_4=I_{16}$, which makes the elements in $
\{P_{i}P_{j}x~|~ i, j=0,1,...,4, i<j \}$ perpendicular to each
other. For example, we have
\begin{eqnarray*}
&&\langle P_0P_1x, P_0P_2x\rangle=\langle P_1x, P_2x\rangle=\langle P_1, P_2\rangle\langle x, x\rangle=0 \\
&&\langle P_1P_2x, P_3P_4x\rangle=\langle P_1P_2x, P_0P_1P_2P_3P_4P_3P_4x\rangle=-\langle P_0x, x\rangle=0.
\end{eqnarray*}
As indicated in \cite{TY}, $Span\{P_iP_jx~|~ i,j=0,1,...,4,
i<j\}\subset T_xM_+$. Observing that $\#  \{P_{i}P_{j}x~|~
i,j=0,1,...,4, i<j \}=\dim T_xM_+=10$, we find that $\{P_iP_jx~|~
i,j=0,1,...,4, i<j \}$ constitutes an orthonormal basis of $T_xM_+$.

Hence, based on the formula (\ref{Ricci}), a fundamental argument in
linear algebra shows immediately that $M_+$ corresponding to $q=2$
is Einstein!

According to \cite{FKM}, this example corresponding to $q=2$ is the
homogeneous one. By the classification of homogeneous hypersurfaces
in spheres (\emph{cf.} \cite{TT}), $M_+$ in this case is
diffeomorphic to $Sp(2)$. \vspace{4mm}

2). $q=0$. On this condition, we have
$\mathrm{Trace}(P_0P_1P_2P_3P_4)=0$. Then there exists $T\in O(16)$,
such that $P_0P_1P_2P_3P_4=T^{t}\left(
\begin{array}{cc}
I&0\\
0&-I
\end{array}
\right)T$. Suppose $M_+$ is Einstein in this case. Then it is
obvious that $\{P_iP_jx~|~ i,j=0,1,...,4, i<j \}$ forms an
orthonormal basis of $T_xM_+$ for any $x\in M_+$. Hence, a simple
verification leads to
\begin{eqnarray*}
&&\langle P_0P_1P_2P_3x, x\rangle=-\langle P_0P_1x, P_2P_3x\rangle=0,\\
&&\langle P_0P_1P_2P_3x, P_iP_jx\rangle=0,\ for\ 0\leq i < j \leq 4,
\end{eqnarray*}
which imply that $P_0P_1P_2P_3x\in T^{\perp}_xM_+$.
Moreover, from the identities
\begin{equation*}
\langle P_0x, P_0P_1P_2P_3x\rangle=\langle P_1x, P_0P_1P_2P_3x\rangle=\langle P_2x, P_0P_1P_2P_3x\rangle=\langle P_3x, P_0P_1P_2P_3x\rangle=0,
\end{equation*}
it follows that $P_0P_1P_2P_3x=\pm P_4x$. That is to say,
\begin{equation}\label{T}
T^{t}\left(
\begin{array}{cc}
I&0\\
0&-I
\end{array}
\right)Tx=\pm x.
\end{equation}
Write $Tx=(y^t,z^t)^t$. Substituting into (\ref{T}), it follows
directly that $y=0$ or $z=0$. In other words, we get a map $T:
M_+^{10} \rightarrow S^{15}$ mapping $x$ to $Tx$  with
$T(M_+^{10})\subset S^7(1)\times\{0\}\cup\{0\}\times S^7(1)$, which
contradicts the fact that $T$ is an orthogonal matrix.

In conclusion, $M_+$ is not Einstein in this case!

\begin{rem}
According to \cite{FKM}, this example corresponding to $q=0$ is the
inhomogeneous one, which is congruent to the unique example of
FKM-type with multiplicities $(3,4)$. Thus $M_+$ of this example is
isometric to $M_-$ of FKM-type with multiplicities $(3,4)$, which is
not Einstein by the proved result in Section 3.1.
\end{rem}
\vspace{1mm}

\subsubsection{\textbf{The (7,8) case.}}
The key point for the proof in this case is an interesting lemma as
we state below, which relates to the condition (A) introduced by
Ozeki and Takeuchi \cite{OT}. We remark that this condition (A) was
interpreted as a condition on the second fundamental form by
\cite{FKM}.
\begin{lem}
Let $M$ be a submanifold minimally immersed in the unit sphere. If $M$ satisfies
condition (A), \emph{i.e.} at some point of $M$, the kernels of all shape operators
$S_N~(N\neq 0)$ coincide, then $M$ is not Einstein.
\end{lem}
\begin{proof}
It is an immediate corollary of Lemma \ref{Einstein}
\end{proof}
From Theorem 5.8 in \cite{FKM}, we see that $M_+$ in the $(7,8)$ case does
satisfy the condition (A). Consequently, $M_+$ is not Einstein.
\vspace{1mm}

\subsubsection{\textbf{The (9,6) case.}}
In this case, $m=9$ and $l=16$. For the Clifford system
$\{P_0,\cdots,P_9\}$ on $\mathbb{R}^{32}$, we choose $x\in S^{31}$
to be a common eigenvector of the commuting operators
$$P_{2i}P_{2i+1}P_{2j}P_{2j+1},\ 0\leq i< j \leq 4.$$ Observing that each
$P_i$ anti-commutes with at least one of these operators, we see
$x\in M_+$. Furthermore, since $x$ is also an eigenvector of the
product of the operators mentioned above, we obtain the following
identities:
\begin{eqnarray*}
&&P_0P_1x=\pm P_2P_3x=\pm P_4P_5x=\pm P_6P_7x=\pm P_8P_9x,\\
&&P_0P_2x=\pm P_1P_3x,\ P_0P_3x=\pm P_1P_2x,\\
&&P_0P_4x=\pm P_1P_5x,\ P_0P_5x=\pm P_1P_4x,\\
&&P_0P_6x=\pm P_1P_7x,\ P_0P_7x=\pm P_1P_6x,\\
&&P_0P_8x=\pm P_1P_9x,\ P_0P_9x=\pm P_1P_8x,\\
&&P_2P_4x=\pm P_3P_5x,\ P_2P_5x=\pm P_3P_4x,\\
&&P_2P_6x=\pm P_3P_7x,\ P_2P_7x=\pm P_3P_6x,\\
&&P_2P_8x=\pm P_3P_9x,\ P_2P_9x=\pm P_3P_8x,\\
&&P_4P_6x=\pm P_5P_7x,\ P_4P_7x=\pm P_5P_6x,\\
&&P_4P_8x=\pm P_5P_9x,\ P_4P_9x=\pm P_5P_8x,\\
&&P_6P_8x=\pm P_7P_9x,\ P_6P_9x=\pm P_7P_8x.
\end{eqnarray*}
As a direct result, $\dim Span\{P_iP_jx~|~ i,j=0,1,...,9, i<j\}\leq\
21= \dim T_xM_+$. Moreover, from the formula (\ref{Ricci}) of Ricci
curvature of $M_+$, for any $X\in T_xM_+$, we derive that
\begin{eqnarray*}
Ric(X)&=& 10+2\Big\{5\langle X, P_{0}P_{1}x\rangle^2+2\sum_{i=2}^9\langle X, P_{0}P_{i}x\rangle^2+2\sum_{i=4}^9\langle X, P_{2}P_{i}x\rangle^2 \\
  && +2\sum_{i=6}^9\langle X, P_{4}P_{i}x\rangle^2+2\sum_{i=8}^9\langle X, P_{6}P_{i}x\rangle^2\Big\},
\end{eqnarray*}
At last, by the following lemma in linear algebra, $M_+$ is not
Einstein in this case.

\begin{lem}
Let $\{u_1,\cdots,u_{21}\}$ be a class of unit vectors in
$\mathbb{R}^{21}$. Then $\rho (X):=5\langle X,
u_1\rangle^2+2\displaystyle\sum_{i=2}^{21}\langle X, u_i\rangle^2$
is not constant on the unit sphere in $\mathbb{R}^{21}$.
\end{lem}

\subsubsection{\textbf{The (8,7) case.}}
In this case, $m=8$ and $l=16$. By the representation theory of
Clifford algebra, we can extend a Clifford system
$\{P_0,\cdots,P_8\}$ on $\mathbb{R}^{32}$ to a system
$\{P_0,\cdots,P_9\}$. Let $x\in S^{31}$ again be a common
eigenvector of the commuting operators
$$P_{2i}P_{2i+1}P_{2j}P_{2j+1},\ 0\leq i< j \leq 4.$$ Repeating the
arguments in 3.2.3, we see that $x$ belongs to $M_+$ and
$$\dim~Span\{P_iP_jx~|~ i,j=0,1,...,9, i<j\}\leq~ 21<22=\dim T_xM_+.$$
Again, using the formula (\ref{Ricci}) for the Ricci curvature, we conclude that $M_+$ is not Einstein in this case.

In fact, there are two incongruent families of FKM-type with multiplicities $(8,7)$,
neither is congruent to the $(7, 8)$ family, and both $M_+$ of the two families are not Einstein.

\subsubsection{\textbf{The (10,21) case.}}
In this case, $m=10$ and $l=32$. Choosing $x\in S^{63}$ to be a common eigenvector of the commuting operators
$$P_0P_1P_2P_3,\ P_4P_5P_6P_7,\ P_0P_1P_8P_9,\ P_2P_3P_8P_9,\ P_0P_2P_8P_{10},$$
we see that $x\in M_+$. Thus a similar argument as in 3.2.3 implies
that
$$\dim~Span\{P_iP_jx~|~ i,j=0,1,...,10, i<j\}<\ 52=\dim T_xM_+.$$
Therefore, $M_+$ is not Einstein in this case. \hfill $\Box$
\vspace{4mm}

Up to now, the proof of 1) in Theorem \ref{thm2} is complete!

\section{Homogeneous isoparametric hypersurfaces}

This section will be committed to investigating the homogeneous cases with multiplicities $(2,2)$ and $(4,5)$.

We begin by recalling a formulation of the Cartan-M{\"u}nzner
polynomial $F$ in terms of the second fundamental forms of the focal
submanifolds, developed by Ozeki and Takeuchi in \cite{OT} (see also
\cite{CCJ}, p.52). For $x\in M_+$, and an orthonormal basis
$\{n_0,n_1,\cdots,n_{m_1} \}$ of the normal space $T^{\perp}_xM_+$,
we introduce the quadratic homogeneous polynomials
\begin{equation}\label{p_i}
p_i(y):=\langle S_{n_i}y, y \rangle,\quad 0\leq i \leq m_1,
\end{equation}
where $y\in T_xM_+$. The Cartan-M{\"u}nzner polynomial F can be expressed by $p_i$ as follows,
\begin{eqnarray}\label{expansion formula}
F(tx+y+w)&=&t^4+(2|y|^2-6|w|^2)t^2+8(\sum_{i=0}^{m_1}p_i(y)w_i)t \nonumber\\
      & & +|y|^4-2\sum_{i=0}^{m_1}(p_i(y))^2+8\sum_{i=0}^{m_1}q_i(y)w_i \\
      & & +2\sum_{i,j=0}^{m_1}\langle\nabla p_i, \nabla p_j\rangle w_i w_j-6|y|^2|w|^2+|w|^4 ,\nonumber
\end{eqnarray}
where the homogeneous polynomials of degree three, $q_i(y)$, are the
components of the third fundamental form of $M_+$,
$w=\displaystyle\sum_{i=0}^{m_1}w_in_i$. \vspace{3mm}

\subsection{The homogeneous example of multiplicities (2,2).}
Consider the lie algebra $so(5, \mathbb{R})$. The special orthogonal
group $SO(5)$ acts on it by the adjoint representation
$$g\cdot Z=gZg^{-1}$$
for $g\in SO(5)$ and $Z\in so(5, \mathbb{R})$. Then the principal
orbits of this action constitute the homogeneous 1-parameter family
of isoparametric hypersurfaces in $S^9$ with multiplicities $(m_1,
m_2)=(2,2)$. Denote the $(i,j)$-entry of $Z$ by $a_{ij}\in
\mathbb{R}$. Then the Euclidean space $\mathbb{R}^{10}$ is $so(5,
\mathbb{R})$ coordinated by $a_{ij}$ with $i<j$, and according to
\cite{OT}, the Cartan-M{\"u}nzner polynomial is
\begin{eqnarray}\label{2,2F}
F(Z)&=&\frac{3}{4}(\mathrm{Trace}Z^2)^2-2\mathrm{Trace}(Z^4) \\
      &=& -\frac{5}{4}\sum_{i}|Z_i|^4 +\frac{3}{2}\sum_{i<j}|Z_i|^2|Z_j|^2-4\sum_{i<j}\langle Z_i, Z_j\rangle^2,\nonumber
      \end{eqnarray}
where $Z_i=(a_{i1},\cdots ,a_{i5})$ is the row vector of $Z=(a_{ij})$.

Making use of the expansion formula (\ref{expansion formula}), we will calculate the second fundamental forms for the
focal submanifolds $M_+:=F^{-1}(1)\cap S^9$ and $M_-:=F^{-1}(-1)\cap S^9$ at some special points,  respectively.
\vspace{2mm}

\noindent
1). The geometry of $M_-$.

Let $e$ be the point in $so(5, \mathbb{R})$ coordinated by
\begin{equation}
\left\{ \begin{array}{ll}
a_{12}=-a_{21}=1, \nonumber\\
a_{ij}=0,\ otherwise.\nonumber
\end{array}\right.
\end{equation}
Substituting $e$ into (\ref{2,2F}), we get immediately that $F(e)=-1$.
It is easy to see that the isotropy subgroup at $e$ is $SO(2)\times SO(3)$,
thus $M_-$ is diffeomorphic to $\widetilde{G}_2(\mathbb{R}^5)=\frac{SO(5)}{SO(2)\times SO(3)}$.

Taking $e$ as the reference point, we can expand the polynomial $-F$ with respect to $a_{12}$:
\begin{eqnarray}
-F&=&a_{12}^4 + a_{12}^2\Big\{2(a_{13}^2+a_{14}^2+a_{15}^2+a_{23}^2+a_{24}^2+a_{25}^2)-6(a_{34}^2+a_{35}^2+a_{45}^2)\Big\} \nonumber\\
      & &+16a_{12}\Big\{a_{34}(a_{24}a_{13}-a_{23}a_{14})+a_{35}(a_{25}a_{13}-a_{23}a_{15})+a_{45}(a_{25}a_{14}-a_{24}a_{15})\Big\}\nonumber\\
      & &+G, \nonumber
\end{eqnarray}
where $G$ denotes the sum of other items containing no
$a_{12}$. Comparing this expansion with (\ref{expansion formula}),
we find that $\{a_{34}, a_{35}, a_{45}\}$ and $\{a_{13}, a_{14},
a_{15},$ $a_{23}, a_{24}, a_{25}\}$ are respectively the normal and
tangent coordinates. Setting $$w_0=a_{34}, w_1=a_{35}, w_2=a_{45},$$
we have from (\ref{expansion formula}) that
\begin{eqnarray*}
p_0&=&2(a_{24}a_{13}-a_{23}a_{14}), \\
p_1 &=& 2(a_{25}a_{13}-a_{23}a_{15}),\\
p_2 &=&2(a_{25}a_{14}-a_{24}a_{15}).
\end{eqnarray*}
Furthermore, using (\ref{p_i}), a direct calculation leads to
$$|S_{0}X|^2+|S_1X|^2+|S_2X|^2=2(a_{13}^2+a_{14}^2+a_{15}^2+a_{23}^2+a_{24}^2+a_{25}^2),$$
for any unit tangent vector $X=(a_{13}, a_{14},
a_{15},$ $a_{23}, a_{24}, a_{25})$. From Lemma \ref{Einstein} and the homogeneity of $M_-$, it follows
that $M_-$ is Einstein.

\begin{rem}
Noticing that $M_-$ is diffeomorphic to the Grassmann manifold $\widetilde{G}_2(\mathbb{R}^5)$
of the oriented two-planes in $\mathbb{R}^5$, it is natural to ask that if the induced metric on $M_-$
from the Euclidean space $\mathbb{R}^{10}$ is the unique invariant metric on the compact irreducible symmetric space
$\widetilde{G}_2(\mathbb{R}^5)$. The answer is affirmative: in virtue of Lemma 1.8
in \cite{SOL}, $M_-\subset \mathbb{R}^{10}$ is just the standard
Pl{\"u}cker embedding of $\widetilde{G}_2(\mathbb{R}^5)$ into $\mathbb{R}^{10}$.
\end{rem}

\vspace{2mm}

\noindent
2). The geometry of $M_+$.

Choose a point $e'$ with coordinates
$a_{12}=a_{34}=\frac{1}{\sqrt{2}}$ and zero otherwise. Clearly,
$F(e')=1$. By a computation of the isotropy subgroup at $e'$
(\emph{cf}. \cite{TXY}), we see that $M_+$ is diffeomorphic to
$\frac{SO(5)}{U(2)}\cong\mathbb{C}P^3$. In terms of new coordinates we
introduced below
$$a_{12}:=(t+w_0)/\sqrt{2},\ a_{34}:=(t-w_0)/\sqrt{2},$$
$$a_{13}:=(w_2-z_2)/\sqrt{2},\ a_{24}:=(w_2+z_2)/\sqrt{2},$$
$$a_{14}:=(z_1-w_1)/\sqrt{2},\ a_{23}:=(z_1+w_1)/\sqrt{2},$$
and
$$x_1:=a_{35},\ x_2:=a_{45},\ y_1:=a_{15},\ y_2:=x_{25},$$
$F$ can be expanded with respect to $t$ as
\begin{eqnarray}
F&=&t^4+ \Big\{2(x_1^2+x_2^2+y_1^2+y_2^2+z_1^2+z_2^2)-6(w_0^2+w_1^2+w_2^2)\Big\}t^2 \nonumber\\
      & &+8\Big\{(x_1^2+x_2^2-y_1^2-y_2^2)w_0+2(x_1y_1+x_2y_2)w_1+2(x_2y_1-x_1y_2)w_2\Big\}t+G',\nonumber
\end{eqnarray}
where $G'$ denotes the sum of other items containing no $t$.
Comparing this expansion with (\ref{expansion formula}), we find that $\{w_0, w_1, w_2\}$ and
$\{x_1, x_2, y_1, y_2, z_1, z_2\}$ are respectively the normal and tangent coordinates.
The components of the second fundamental form of $M_+$ at $e'$ are
\begin{eqnarray*}
p_0&=&x_1^2+x_2^2-y_1^2-y_2^2,\\
p_1&=&2(x_1y_1+x_2y_2),\\
p_2&=&2(x_2y_1-x_1y_2).
\end{eqnarray*}

Therefore, for any $X=(x_1, x_2, y_1, y_2, z_1, z_2)\in T_{e'}M_+$ with $|X|=1$,
$$|S_{0}X|^2+|S_1X|^2+|S_2X|^2=3(x_1^2+x_2^2+y_1^2+y_2^2),$$
which is not constant, namely, $M_+$ is not Einstein by Lemma \ref{Einstein}.
\vspace{2mm}

\subsection{The homogeneous example of multiplicities (4,5).}
This case resembles the $(2,2)$ case strongly. However, both focal submanifolds $M_+$ and $M_-$ in this case
are not Einstein. To show the assertion, consider the Lie algebra $so(5, \mathbb{C})$. The unitary group $U(5)$
acts on it by the adjoint representation
$$g\cdot Z=\overline{g}Zg^{-1}$$
for $g\in U(5)$ and $Z\in so(5, \mathbb{C})$. The principal orbits of this action constitute the
homogeneous 1-parameter family of isoparametric hypersurfaces in $S^{19}$ with multiplicities $(m_1, m_2)=(4,5)$.
Denote the $(i,j)$-entry of $Z$ by $a_{ij}=x_{ij}+\sqrt{-1}y_{ij}$ with real $x_{ij}$ and $y_{ij}$. Then the Euclidean space $\mathbb{R}^{20}$ is $so(5, \mathbb{C})$ coordinated by $x_{ij}$ and $y_{ij}$ with $i<j$.
Again, according to \cite{OT}, the Cartan-M{\"u}nzner polynomial is
\begin{eqnarray}
F(Z)&=&\frac{3}{4}(\mathrm{Trace}Z\overline{Z})^2-2\mathrm{Trace}(Z\overline{Z})^2, \nonumber\\
      &=& -\frac{5}{4}\sum_{i}|Z_i|^4 +\frac{3}{2}\sum_{i<j}|Z_i|^2|Z_j|^2-4\sum_{i<j}|\langle Z_i, Z_j\rangle|^2,\nonumber
      \end{eqnarray}
where $Z_i=(a_{i1},\cdots ,a_{i5})$ is the row vector of $Z=(a_{ij})$ and $\langle Z_i, Z_j\rangle$ is the Hermitian inner product .

Comparing with the expansion formula (\ref{expansion formula}), we will investigate the second fundamental forms for the
focal submanifolds $M_+:=F^{-1}(1)\cap S^{19}$ and $M_-:=F^{-1}(-1)\cap S^{19}$ respectively.
\vspace{2mm}

\noindent
1). Geometry of $M_+$.

Following \cite{CHI}, we choose a point $e$ with coordinates $x_{12}=x_{34}=\frac{1}{\sqrt{2}}$
and zero otherwise. Clearly, $F(e)=1$.
In a similar way as 2) in 4.1, we introduce new coordinates
\begin{eqnarray*}
&& x_{12}:=(t+w_0)/\sqrt{2},\ x_{34}:=(t-w_0)/\sqrt{2}, \\
&& x_{13}:=(w_3-z_4)/\sqrt{2},\ x_{24}:=(w_3+z_4)/\sqrt{2}, \\
&& y_{13}:=(-z_3-w_4)/\sqrt{2},\ y_{24}:=(-z_3+w_4)/\sqrt{2}, \\
&& x_{14}:=(z_2-w_1)/\sqrt{2},\ x_{23}:=(z_2+w_1)/\sqrt{2}, \\
&& y_{14}:=(w_2+z_1)/\sqrt{2},\ y_{23}:=(w_2-z_1)/\sqrt{2},
\end{eqnarray*}
and
\begin{eqnarray*}
&& x_1:=x_{35},\ x_2:=y_{35},\ x_3:=x_{45},\ x_4:=y_{45},\ x_5:=y_{34},\\
&& y_1:=x_{15},\ y_2:=y_{15},\ y_3:=x_{25},\ y_4:=y_{25},\ y_5:=y_{12}.
\end{eqnarray*}
Then $(w_0, w_1, w_2, w_3, w_4)$ are the normal coordinates, $(x_1,...,x_5,y_1,...,y_5,z_1,...,z_4)$
are the tangent coordinates. Expanding $F$ with respect to $t$, the components of the second fundamental form of $M_+$ at $e$  are given by
\begin{eqnarray*}
&&p_0=x_{1}^2+\cdots +x_5^2-y_1^2-\cdots-y_5^2,\\
&&p_1=2(x_1y_1+\cdots+x_4y_4) +\sqrt{2}(x_5+y_5)z_1,\\
&&p_2=2(x_2y_1-x_1y_2)+2(x_3y_4-x_4y_3)+\sqrt{2}(x_5+y_5)z_2,\\
&&p_3=2(x_3y_1-x_1y_3)+2(x_4y_2-x_2y_4)+\sqrt{2}(x_5+y_5)z_3,\\
&&p_4=2(x_2y_3-x_3y_2)+2(x_4y_1-x_1y_4)+\sqrt{2}(x_5+y_5)z_4.
\end{eqnarray*}
Therefore, for any $X=(x_1, \cdots, x_5, 0,\cdots, 0, 0, \cdots, 0)\in T_eM_+$,
$$\sum_{i=0}^{4}|S_{i}X|^2=5\sum_{\alpha  =1}^{5}x_{\alpha }^2-2x_5^2,$$
which implies that $M_+$ is not Einstein by Lemma \ref{Einstein}. We remark that $M_+$ is diffeomorphic to the homogeneous space $\frac{U(5)}{Sp(2)\times U(1)}$ (\emph{cf.} \cite{TXY}). In fact, it fibers over $\mathbb{C}P^4$ with fiber $\frac{U(4)}{Sp(2)}$.
\vspace{2mm}

\noindent
2). Geometry of $M_-$.

Let $e'$ be a point with coordinates $x_{12}=-x_{21}=1$ and zero otherwise. Clearly, we have $F(e')=-1$.
With respect to $x_{12}$, $-F$ can be expanded as:
$$-F=x_{12}^4+Ax_{12}^2+8Bx_{12}+C,$$
where
\begin{eqnarray}
A&=&2(y_{12}^2+|a_{13}|^2+|a_{14}|^2+|a_{15}|^2+|a_{23}|^2+|a_{24}|^2+|a_{25}|^2)\nonumber\\
 & & -6(|a_{34}|^2+|a_{35}|^2+|a_{45}|^2),\nonumber\\
B&=&x_{34}(-2x_{14}x_{23}+2x_{13}x_{24}+2y_{14}y_{23}-2y_{13}y_{24}) \nonumber\\
 & &+x_{35}(-2x_{15}x_{23}+2x_{13}x_{25}+2y_{15}y_{23}-2y_{13}y_{25})\nonumber\\
 & &+x_{45}(-2x_{15}x_{24}+2x_{14}x_{25}+2y_{15}y_{24}-2y_{14}y_{25})\nonumber\\
 & &+y_{34}(-2x_{14}x_{23}+2x_{13}y_{24}-2y_{14}x_{23}+2y_{13}x_{24})\nonumber\\
 & &+y_{35}(-2x_{15}y_{23}+2x_{13}y_{25}-2y_{15}x_{23}+2y_{13}x_{25})\nonumber\\
 & &+y_{45}(-2x_{15}y_{24}+2x_{14}y_{25}-2y_{15}x_{24}+2y_{14}x_{25}),\nonumber
\end{eqnarray}
$C$ denotes the sum of those items containing no $x_{12}$.
From this expansion formula, we see that $\{x_{34},\ x_{35},\ x_{45},\ y_{34},\ y_{35},\ y_{45}\}$
and $\{y_{12},\ x_{13},\ y_{13},$ $x_{14},\ y_{14},\ x_{15},\ y_{15},\ x_{23},$ $y_{23},\ x_{24},\ y_{24},\ x_{25},\ y_{25}\}$ are respectively the normal and tangent coordinates.
Setting $$w_0=x_{34},\ w_1=x_{35},\ w_2=x_{45},\ w_3=y_{34},\ w_4=y_{35},\ w_5=y_{45},$$
we get
\begin{eqnarray*}
p_0&=& -2x_{14}x_{23}+2x_{13}x_{24}+2y_{14}y_{23}-2y_{13}y_{24},\\
p_1 &=& -2x_{15}x_{23}+2x_{13}x_{25}+2y_{15}y_{23}-2y_{13}y_{25}, \\
p_2 &=& -2x_{15}x_{24}+2x_{14}x_{25}+2y_{15}y_{24}-2y_{14}y_{25}, \\
p_3 &=& -2x_{14}x_{23}+2x_{13}y_{24}-2y_{14}x_{23}+2y_{13}x_{24}, \\
p_4 &=& -2x_{15}y_{23}+2x_{13}y_{25}-2y_{15}x_{23}+2y_{13}x_{25}, \\
p_5 &=& -2x_{15}y_{24}+2x_{14}y_{25}-2y_{15}x_{24}+2y_{14}x_{25}.
\end{eqnarray*}
Observing that the unit tangent vector $Y_{12}$ with $y_{12}=1$ is contained in $Ker(S_{i})$ for $0\leq i\leq 5$,
we derive from the formulas above that
$$\sum_{i=0}^5|S_{i}(Y_{12})|^2=0.$$
On the other hand, there does exist a unit tangent vector $v$ such that
$$\sum_{i=0}^5|S_{i}v|^2\neq 0.$$
This leads us to the final conclusion that
$M_-$ in the $(4, 5)$ case is not Einstein by Lemma \ref{Einstein}. We remark that $M_+$ is diffeomorphic to $\frac{U(5)}{SU(2)\times U(3)}$ (\emph{cf.} \cite{TXY}).

\begin{ack}
The authors would like to thank Professor Q.S.Chi for his excellent lectures on isoparametric
hypersurfaces given at Beijing Normal University and valuable discussions. They also express their sincere gratitude to Professor T.E.Cecil for his
handwritten translation of \cite{FKM} and many helpful comments during the third author preparing the Latex version of
the manuscript.
\end{ack}

\end{document}